\documentclass[12pt,a4paper]{amsart}
\usepackage{amssymb}
\usepackage{amsmath}
\usepackage{amsfonts}
\usepackage[english]{babel}
\usepackage[T1]{fontenc}
\usepackage[latin1]{inputenc}
\usepackage{amsthm}
\usepackage[all]{xy}
\usepackage{a4wide}
\usepackage{amssymb}
\usepackage{graphicx}

\date{\today}
\theoremstyle{plain}
\newtheorem{theorem}[equation]{Theorem}

\newtheorem{prop}[equation]{Proposition}
\newtheorem{cor}[equation]{Corollary}

\theoremstyle{definition}

\newtheorem{ex}[equation]{Example}
\newtheorem{remark}[equation]{Remark}
\theoremstyle{remark}

\theoremstyle{remarks}

\theoremstyle{definition}

\theoremstyle{remark}

\numberwithin{equation}{section}
\numberwithin{thm}{section}

\newcommand{\N}{\mathbb{N}}

\newcommand{\A}{\mathfrak{A}}
\newcommand{\SUM}{\sum_{k=1}^\infty}

\begin{document}

\title[quotient algebra $\mathcal{K}(X)/\mathcal{A}(X)$]{Quotient algebra 
of compact-by-approximable operators on Banach spaces failing the approximation property}

\author[Hans-Olav Tylli and Henrik Wirzenius]{Hans-Olav Tylli and Henrik Wirzenius}
\address{Tylli: Department of Mathematics and Statistics, Box 68,
FI-00014 University of Helsinki, Finland}
\email{hans-olav.tylli@helsinki.fi}
\address{Wirzenius: Department of Mathematics and Statistics, Box 68,
FI-00014 University of Helsinki, Finland}
\email{henrik.wirzenius@helsinki.fi}


\begin{abstract}
We initiate a study of structural properties of the quotient algebra 
$\mathcal K(X)/\mathcal A(X)$ of the compact-by-approximable operators  
on Banach spaces $X$ failing the approximation property. Our main results and examples 
include the following: 
(i) there is a linear isomorphic embedding from $c_0$ into $\mathcal K(Z)/\mathcal A(Z)$,
where $Z$ belongs to the class of Banach spaces constructed by Willis 
that have the metric compact approximation
property but fail the approximation property, (ii) there is a linear isomorphic embedding
from a non-separable space $c_0(\Gamma)$ into $\mathcal K(Z_{FJ})/\mathcal A(Z_{FJ})$, 
where $Z_{FJ}$ is a universal compact factorisation space arising from the work of 
Johnson and Figiel.
\end{abstract}

\maketitle

\section{Introduction}\label{intro}

For Banach spaces  $X$ and $Y$ let $\mathcal{K}(X,Y)$ 
be the Banach space of compact operators $X \to Y$ in the uniform operator norm, 
and denote by $\mathcal{A}(X,Y) = \overline{\mathcal{F}(X,Y)}$ the closed subspace consisting of the approximable operators $X \to Y$. 
Here $\mathcal{F}(X,Y)$ is the linear subspace consisting 
of the bounded finite rank operators $X \to Y$.
We put $\mathcal K(X) = \mathcal K(X,X)$ and 
$\mathcal{A}(X) =  \mathcal{A}(X,X)$ for $Y = X$, so that
 $\mathcal{A}(X) \subset  \mathcal{K}(X)$ are closed two-sided ideals of the Banach algebra 
 $\mathcal L(X)$ of the bounded operators on $X$. Consequently
 the quotient algebra $\mathfrak A_X := \mathcal K(X)/\mathcal A(X)$ 
 is a non-unital Banach algebra 
 equipped with the quotient norm
\[
\Vert T+ \mathcal A(X) \Vert = dist(T,\mathcal A(X))=\inf_{S\in \mathcal A(X)}\Vert T-S\Vert .
 \]

It is well known that $\mathcal{A}(X) = \mathcal{K}(X)$ whenever $X$ has the 
approximation property, so that the quotient $\mathfrak A_X$ 
can only be non-zero within the class of  Banach spaces $X$ that fail to have the approximation property. It remains a longstanding problem, see \cite[Problem 1.e.9]{LT1} or 
\cite[Problem 2.7]{C}, whether conversely $\mathcal{A}(X) = \mathcal{K}(X)$ will imply that
$X$ has the approximation property. 
Properties of the quotient algebras $\mathfrak A_X$ are rather
elusive and inaccessible, not least because it is a
non-trivial task to construct examples of Banach spaces failing the approximation property,
and they have not been much studied.
Recently Dales \cite{D13} revived the interest in various questions about the size 
and algebraic structure of 
$\mathfrak A_X$ by collecting results and highlighting a number of problems.
It is known, see \cite[2.5.8(iv)]{D00}, that $\mathfrak A_X$ 
 is a radical Banach algebra for any Banach space $X$, that is, any quotient element
 $S + \mathcal{A}(X)$ is quasi-nilpotent in $\mathfrak A_X$. 
 Consequently, from an algebraic perspective the non-trivial 
 quotient algebras $\mathfrak A_X$  are natural (non-commutative)  
 radical Banach algebras whose structure is very poorly understood.

It was shown by Bachelis \cite{B} that if 
$E$ is a Banach space which has  the bounded approximation property
and $E$ contains a closed linear subspace $X$ which fails the approximation property,
then $E$ contains a closed linear subspace $Y$ such that 
$\mathcal{A}(Y, X) \varsubsetneq \mathcal{K}(Y, X)$.
Moreover, if in addition $E \oplus E$ is isomorphic to $E$, 
then $E$ contains the closed linear subspace $Z = Y \oplus X$ for which
$\mathcal{A}(Z) \varsubsetneq \mathcal{K}(Z)$.
This extended an earlier result due to Alexander \cite{A}
for the sequence spaces $E = \ell^p$ with $2 < p < \infty$.

In this paper we show that 
in fact the quotient algebra $\mathfrak A_X$ is  infinite-dimensional 
for many Banach spaces $X$. Our results mostly
draw on techniques from Banach space theory, which
provide less information about the algebraic structure of $\mathfrak A_X$.
In section \ref{general} we review for our purposes two general constructions, which yield
that the quotient algebra $\mathfrak A_X$ is infinite-dimensional for certain classes 
of Banach spaces. 
In section \ref{Willis} we first observe that $\mathfrak A_X$ is infinite-dimensional 
for the class of Banach spaces $X$, where $X$ has the bounded  compact approximation 
property but fails the approximation property. The first examples of such Banach spaces were constructed by Willis \cite{W}, and our main result demonstrates 
more precisely that $\mathfrak A_Z$ always contains a linear isomorphic copy of the 
sequence space $c_0$ for the spaces $Z$ from \cite{W}.
However, the linear embedding $c_0 \to \mathfrak A_Z$ does not 
preserve much of the algebraic structure.
In section \ref{further} we further show that $\mathfrak A_X$  is infinite-dimensional 
for two different universal spaces $X$. In the first example 
$X = C_1^*$, where $C_1$ is the complementably universal separable conjugate space 
constructed by Johnson \cite{J72}. In the second example  we show that there is a linear isomorphic embedding of a non-separable sequence space $c_0(\Gamma)$ into  
$\mathfrak A_{Z_{FJ}}$,  where $Z_{FJ}$ is a universal compact factorisation space 
suggested by results of Johnson \cite{J71} and Figiel \cite{F}. 

Recently there has been significant advances in the explicit classification of the
lattice of closed two-sided ideals $I \subset \mathcal L(X)$ for some classical Banach spaces $X$,
see e.g. \cite{SZ}, \cite{FSZ} and their references. Moreover, 
the Calkin algebra $\mathcal L(X)/\mathcal K(X)$ has been explicitly determined 
for special Banach spaces $X$, see \cite{AH}, \cite{Td}, \cite{MPZ} and their references. 
However, the preceding results concern Banach spaces having a Schauder basis, or at least
the bounded approximation property, which are  disjoint classes of spaces from
those relevant for the study of the quotients $\mathfrak A_X$. 

\section{General constructions}\label{general}

In this section we review for our subsequent use two 
direct sum constructions (perhaps known in some form) that produce Banach spaces $Z$ 
for which the quotient algebra $\mathfrak A_Z$ is infinite-dimensional, by starting 
from any Banach space $X$ that fails the approximation property. The drawback 
is that in general $X$ only embeds as a complemented subspace in $Z$,
where $Z$ is typically quite different from $X$. We also discuss 
the algebraic relevance of these constructions. 

We first recall the various approximation properties of Banach spaces 
that will be required. 
The Banach space  $X$ has the approximation property (A.P. in short) 
if for any compact subset $K\subset X$ and any $\varepsilon>0$ 
there is a finite rank operator $T\in \mathcal F(X)$ such that
\begin{equation}\label{ap0}
\sup_{x\in K} \Vert x-Tx \Vert \leq\varepsilon. 
\end{equation}
If one instead allows approximation by $T \in \mathcal{K}(X)$ in condition \eqref{ap0}, then
$X$ is said to have the compact approximation property (C.A.P.).
 Moreover, $X$ has the bounded approximation property (B.A.P.),
 respectively the bounded compact approximation property (B.C.A.P.), if there is a  
 uniform constant $M < \infty$ so that  the approximating operator $T$ from \eqref{ap0} 
 can be chosen to satisfy $\Vert T\Vert \leq M$.
Finally, $X$ has the metric C.A.P. if above $M = 1$.
 
 We refer e.g. to \cite[1.e]{LT1} and the survey \cite{C} for useful general information 
 about these approximation properties. 
 Recall that any Banach space with a Schauder basis has the B.A.P., and 
 that the first example of a Banach space without the A.P. 
 was constructed by Enflo \cite{E}.
The references \cite[2.d]{LT1}, \cite[1.g]{LT2} and \cite[10.4]{Pie}
contain the detailed constructions
of some Banach spaces failing the A.P.
Presently the space $\mathcal L(\ell^2)$ of the bounded operators on the Hilbert space
$\ell^2$ is the most explicit space  
known to fail the A.P. by a result of Szankowski \cite{Sz81}.
Moreover, the Calkin algebra $\mathcal L(\ell^2)/\mathcal K(\ell^2)$ also fails the
A.P. \cite{GS}.

To motivate our first construction we recall a classical fact due to
Grothendieck. If the Banach space $X$ fails to have the A.P.,
then there is a linear subspace $Y \subset X$ and a complete norm $\vert \cdot \vert$
on $Y$, such that the inclusion map  $J:  (Y,\vert \cdot \vert) \to X$ 
is a compact non-approximable operator. For this fact see e.g. 
the argument for the implication (v) $\Rightarrow$ (i) in 
\cite[Thm. 1.e.4]{LT1}. It easily follows  that 
$\mathcal A(X \oplus Y) \varsubsetneq \mathcal K(X \oplus Y)$,
where $Y = (Y,\vert \cdot \vert)$. 

This fact can be generalised in many ways, and the
following construction implies that  $\mathfrak A_Z$ is infinite-dimensional, where 
e.g.  $Z$ is the countable direct $\ell^p$-sum
$X \oplus \big(\oplus_{n\in \mathbb N} Y)_{\ell^p}$, and $X$ and $Y$ are as above. 
Here $1 \le p \le \infty$,
where for  notational unity the case $p = \infty$ 
denotes  a direct $c_0$-sum  equipped with the
supremum norm. 

\begin{prop}\label{sum1}
Suppose that $X$ and $Y_n$ ($n \in \mathbb N$) are Banach spaces such that 
one of the following conditions holds:
\begin{equation}\tag{i}
\mathcal A(Y_n,X) \varsubsetneq \mathcal K(Y_n,X), \quad n \in \mathbb N,
\end{equation}
\begin{equation}\tag{ii}
\mathcal A(X,Y_n) \varsubsetneq \mathcal K(X,Y_n), \quad n \in \mathbb N.
\end{equation}
Let $Z = X \oplus \big(\oplus_{n\in \mathbb N} Y_n\big)_{\ell^p}$ for $1 \le p \le \infty$. Then  
the quotient $\mathfrak A_Z$ is infinite-dimensional. 
\end{prop}

\begin{proof}
Let $P_0: Z \to X$ and $P_n: Z \to Y_n$ be the natural projections, respectively
$J_0: X \to Z$ and $J_n: Y_n \to Z$ the corresponding inclusion maps
for $n \in \mathbb N$. 

Suppose that (i) holds.  We may after a normalisation find compact operators
$S_n \in \mathcal K(Y_n,X)$ so that
\[
dist(S_n,\mathcal A(Y_n,X)) = 1, \quad n \in \mathbb N.
\]
Consider the compact operators $T_n = J_0S_nP_n \in \mathcal K(Z)$ 
for $n \in \mathbb N$. We claim that
\begin{equation}\label{diff}
dist(T_n - T_m,\mathcal A(Z)) \ge 1
\end{equation}
for any $n, m \in \mathbb N$ with $n \neq m$. In particular, \eqref{diff}  says that
$(T_n + \mathcal A(Z)) \subset \mathfrak A_Z$ 
is a bounded sequence which 
does not have any convergent subsequences, 
whence   $\mathfrak A_Z$  is infinite-dimensional.

Towards the desired estimate suppose that $V \in \mathcal A(Z)$ is arbitrary. Then
$P_0VJ_n \in \mathcal A(Y_n,X)$, so that for $m \neq n$ one gets that
\begin{align*}
1 = \  & dist(S_n,\mathcal A(Y_n,X)) \le \Vert S_n - P_0VJ_n\Vert \\
= \ & \Vert P_0(T_n - T_m - V)J_n\Vert \le \Vert T_n - T_m - V \Vert,
\end{align*}
since $T_mJ_n = J_0S_mP_mJ_n = 0$ above. This implies that
\eqref{diff} holds.

If condition (ii) holds, then there are  $R_n \in \mathcal K(X,Y_n)$ so that
$dist(R_n,\mathcal A(X,Y_n)) = 1$
for $n \in \mathbb N$. Let $U_n = J_nR_nP_0 \in \mathcal K(Z)$ for $n \in \mathbb N$. 
A straightforward modification of the estimate for  \eqref{diff} shows that 
also in this case
$dist(U_n - U_m,\mathcal A(Z)) \ge 1$
for $n \neq m$. This completes the argument.
\end{proof}

Of course, the argument in Proposition \ref{sum1} also shows that
$\mathcal K(Y,X)/ \mathcal A(Y,X)$, respectively $\mathcal K(X,Y)/ \mathcal A(X,Y)$,
are  infinite-dimensional quotient spaces  for $Y = \big(\oplus_{n\in \mathbb N} Y_n\big)_{\ell^p}$.
However, our focus is on the  quotient algebras $\mathfrak A_X$,
and it must be kept in mind that it is much more difficult  to study  
$\mathfrak A_X$ for a \textit{given} Banach space $X$. 
Recently K\"ursten and Pietsch \cite{KP} 
produced various examples of compact non-approximable
operators as the central part of the canonical factorisation
$\ell^1 \to \ell^1/Ker(T) \to \overline{T(\ell^1)} \to c_0$ for certain operators 
$T \in \mathcal L(\ell^1,c_0)$,
but such methods are less adapted to specific quotient algebras.
 
We will later apply the following variant of Proposition \ref{sum1}, 
which contains more precise information. In conditions (iii) and (iv) below
the $\ell^p$-sum is replaced by the supremum norm and a $c_0$-condition for $p = \infty$.

\begin{prop}\label{sum2}
Suppose that $X$ is a Banach space having the following properties:

(i) There is a uniformly bounded sequence $(P_n) \subset \mathcal L(X)$ of 
projections on $X$ so that  $P_n P_m=0$ whenever $n\neq m$,

(ii)  $\mathfrak A_{X_n}\neq \{0\}$ for every $n\in\N$, where $X_n=P_n X$.

\noindent Then $\mathfrak A_X$ is infinite-dimensional. 

Suppose in addition that 
there is $1 \le p \le \infty$ and constants $C, D < \infty$ so that

(iii) $\Vert \sum_{n=1}^\infty x_n\Vert \le 
C \cdot \big(\sum_{n=1}^\infty \Vert x_n \Vert^p \big)^{1/p}$\\
whenever $x_n \in P_nX$ for $n \in \mathbb N$ and 
$\sum_{n=1}^\infty \Vert x_n\Vert^p < \infty$, 

(iv) $\big(\sum_{n=1}^\infty \Vert P_nx \Vert^p \big)^{1/p} 
\le D \Vert x \Vert$ for  $x \in X$.

\smallskip

Then there is a linear embedding $c_0 \to \A_X$, that is, 
$\A_X$ contains a closed subspace which is linearly isomorphic 
to the sequence space $c_0$.
\end{prop}

Note that  in  (iii) the condition implies that the series $\sum_{n=1}^\infty x_n$ 
converges in $X$.
Moreover,  (iii) is always valid with $C = 1$ for  $p = 1$, 
and (iv) follows from (i) for $p = \infty$. 
The typical application of Proposition \ref{sum2} is to direct $\ell^p$-sums 
$X = \big(\oplus_{n\in \mathbb N} X_n\big)_{\ell^p}$,
where  $\mathfrak A_{X_n}\neq \{0\}$ for every $n\in\N$,
see Corollary \ref{corsum2} below. However, 
it is not assumed here that the linear span of $\cup_{n\in \mathbb N} P_nX$ 
is dense in $X$, which will be relevant in Section \ref{Willis}.

\begin{proof}
Suppose first that $X$ satisfies conditions (i) and (ii), and put
$M = \sup_n \Vert P_n\Vert$.
From (i) we find, after a normalisation and possibly by passing to $S_n - V_n$
for some $V_n \in \mathcal A(X_n)$, compact operators
$S_n \in \mathcal K(X_n)$ so that 
\begin{equation}\label{bd1}
dist(S_n,\mathcal A(X_n)) = 1\quad \textrm{and}\quad \Vert S_n\Vert \le 2
\end{equation}
for  $n \in \mathbb N$.

Let $J_n: X_n \to X$ be the inclusion map for $X_n = P_nX$ and $n \in \mathbb N$.
Consider the compact operators $U_n = J_nS_nP_n \in \mathcal K(X)$ for 
$n \in \mathbb N$. By a straightforward modification of the estimates in Proposition
\ref{sum1} one obtains that 
\[
dist(U_n - U_m,\mathcal A(Z)) \ge 1/M
\]
for any $n, m \in \mathbb N$ with $n \neq m$. For this estimate one uses the fact that 
$P_nP_m = 0$ implies that
$P_nJ_m = 0$,  for $m \neq n$. Thus $\A_X$ is infinite-dimensional.

Suppose next that $X$ also satisfies conditions (iii) and (iv)  for $1 \le p < \infty$.
 Let $(a_k)\in c_0$ be arbitrary. 
 Observe that for any $m>n$ and $x \in X$ we get from (iii)  and (iv) that
 \begin{align*}
\Vert \sum_{k=1}^m a_kU_kx -\sum_{k=1}^n a_kU_kx\Vert = & 
\Vert  \sum_{k=n+1}^m a_kJ_kS_kP_kx \Vert 
\le  C \cdot \big(\sum_{k=n+1}^m  \vert a_k\vert^p \cdot \Vert S_kP_kx\Vert^p\big)^{1/p} \\
& \le \ 2C \cdot \sup_{n+1\leq k\leq m} \vert a_k\vert 
\cdot \big(\sum_{k=n+1}^m\Vert P_kx\Vert^p \big)^{1/p}\\
& \le 2CD \cdot  \sup_{n+1\leq k\leq m} \vert a_k\vert \cdot \Vert x\Vert,
 \end{align*}
which converges uniformly to $0$ as $n\to\infty$.   We conclude that 
$(\sum_{k=1}^n a_kU_k)_{n \in \mathbb N}$ is a Cauchy-sequence in 
$\mathcal K(X)$, so that 
$\sum_{k=1}^\infty a_k U_k \in\mathcal K(X)$ defines a compact operator
for any $(a_k)\in c_0$. Moreover, for $x \in X$ and $(a_k)\in c_0$  
we obtain similarly as above from (iii) and (iv) that
\[
\Vert \sum_{k=1}^\infty a_k U_kx \Vert \le 2CD \cdot   
\sup_{k \in \mathbb N} \vert a_k\vert \cdot \Vert x\Vert,
\]
so  that
 \begin{align*}
 dist(\SUM a_kU_k, \mathcal A(X)) \leq ||\SUM a_kU_k|| 
  \leq \ & 2CD \cdot||(a_k)||_\infty .
 \end{align*}
This means that 
$(a_k)\mapsto (\SUM a_kU_k) + \mathcal A(X)$ 
defines a bounded linear map $\psi: c_0 \to \A_X$. The argument for 
$p = \infty$ is similar.
 
We claim that $\psi$ is bounded from below for $1 \le p \le \infty$, 
so that $\psi$ defines a linear embedding of $c_0$ into $\A_X$. 
Towards this,  let $(a_k)\in c_0$ be arbitrary and pick  $r \in \mathbb N$ 
such that $||(a_k)||_\infty=|a_{r}|$.
Then for any $U\in \mathcal A(X)$ we get from $P_rJ_k = 0$ for 
$r \neq k$, that
\begin{align*}
M\cdot||\SUM a_kU_k-U||&\geq ||P_r(\SUM a_k U_k-U)J_r||=||P_r(\SUM a_k J_kS_kP_k-U)J_r||\\
&=||a_rS_r-P_rUJ_r||\geq ||a_r S_r+ \mathcal A(X_r)|| = |a_r|.
\end{align*}
It follows that 
$||\psi(a_k)||= dist(\sum_{k=1}^\infty a_k U_k, A(X)) \geq (1/M) ||(a_k)||_\infty$
for $(a_k)\in c_0$, which concludes the argument. 
\end{proof}

Note that in the above argument the series  
$\sum_{k=1}^\infty a_k U_k$ defines a bounded 
operator  on $X$ for any $(a_k) \in \ell^\infty$, but this operator will
usually not be compact.

In the sequel we will use the notation $Y \approx Z$ for linearly isomorphic spaces
$Y$ and $Z$.

\begin{cor}\label{corsum2}
Suppose that there is $1 \le p \le \infty$  such that one of the following conditions holds
(where the case $p = \infty$  refers below to a direct  $c_0$-sum):

(i) $X = \big(\oplus_{n\in \mathbb N} X_n\big)_{\ell^p}$,
where  $\mathfrak A_{X_n}\neq \{0\}$ for every $n\in \N$, or

(ii) $X$ is a Banach space so that  $\mathfrak A_{X}\neq \{0\}$ and 
$\big(\oplus_{\mathbb N} X)_{\ell^p} \approx X$.

\noindent Then $\A_X$ contains a linear isomorphic copy of $c_0$.
\end{cor}

The basic examples are found in the following application.

\begin{cor}\label{corsum3} 
Let $1 \leq p < \infty$ and $p \neq 2$. Then there is a closed linear subspace $X \subset \ell^p$
so that $\A_X$ contains a linear isomorphic copy of $c_0$.
Moreover, there is a closed linear subspace $X \subset c_0$, so that 
$\A_X$ contains a copy of $c_0$.
\end{cor}

\begin{proof}
There is a closed linear subspace  $Z \subset \ell^p$ that fails the A.P. for 
any $1 \le p < \infty$ and $p \neq 2$. This follows from the work of Enflo and Davie 
for $2 < p < \infty$, see \cite[2.d]{LT1}, and  Szankowski
\cite{Sz78} for $1 \le p < 2$. 
By \cite[Thm. 1]{B} there is a closed linear subspace $Y\subset \ell^p$
such that  $\mathcal{A}(Y\oplus Z) \varsubsetneq \mathcal{K}(Y\oplus Z)$.
Let 
\[
X = \Big( \oplus_{\mathbb N} (Y \oplus Z) \Big)_{\ell^p} \subset 
\big(\oplus_{\mathbb N} \ell^p)_{\ell^p} \approx \ell^p.
\] 
Then Corollary \ref{corsum2} yields a linear embedding $c_0 \to \A_X$.
The argument for $c_0$ is similar.
\end{proof}

\begin{remark}\label{bachelis}
Suppose that $E$ has  the B.A.P.,  $E \oplus E \approx E$ 
and $E$ contains a closed subspace $X$ which fails the A.P.
By \cite[Thm. 1]{B} there is a closed linear subspace $Y\subset E$
such that  $\mathcal{A}(Y, X) \varsubsetneq \mathcal{K}(Y, X)$. Fix
$n \in \mathbb N$ and  consider
 \[
 Z_n = X \oplus (Y \oplus \ldots \oplus Y),
 \]
with $n$ copies of $Y$ in the direct sum. Then $Z_n$ is 
a closed subspace of $E$, up to linear isomorphism, since $E \oplus E \approx E$.
It is not difficult to modify the argument in Propositions \ref{sum1} and \ref{sum2}
to find a  linearly independent  family
$\{T_j + \mathcal A(Z_n): j = 1,\ldots,n\}$ in $\mathfrak A_{Z_n}$, 
whence $dim(\mathfrak A_{Z_n}) \ge n$.

However, we note that  it is not clear whether in this setting from 
\cite[Thm. 1]{B} there always is a closed linear subspace $Z \subset E$
such that $\A_Z$ is infinite-dimensional.
\end{remark}

The algebraic differences between $c_0$ and radical Banach algebras readily imply that the 
linear embedding $\psi: c_0 \to \mathfrak A_{X}$
in Proposition \ref{sum2}, or  its applications, is not an algebra homomorphism.
In fact, $\psi$  cannot preserve much of the multiplicative structure. 

\begin{prop}\label{prod} 
Suppose that $\psi$ is any linear embedding $c_0 \to \A_X$, where $X$ is a Banach space.
Then for every non-zero $a \in c_0$ there  is $b \in c_0$ such that
$\psi(ab) \neq \psi(a) \psi(b)$.
\end{prop}

\begin{proof}
It is straightforward to check that 
\[
\mathcal A = \{ a \in c_0: \psi(ax) = \psi(a)\psi(x) = \psi(x)\psi(a)\  \textrm{for all}\ x \in c_0\}
\]
is a closed subalgebra of $c_0$. If $\mathcal A \neq \{0\}$, then
the restriction $\psi_{|\mathcal A}$ defines an algebra isomorphism
$\mathcal A \to \mathcal B$,
where $\mathcal B = \psi(\mathcal A)$ is a non-trivial closed subalgebra of 
$\mathfrak A_{X}$. In particular, 
\[
\sigma_{\mathcal A}(a) = \sigma_{\mathcal B}(\psi(a)), \quad a \in \mathcal A,
\]
where $\sigma_{\mathcal A}(a)$ denotes the spectrum of $a  \in \mathcal A$ 
computed in the subalgebra
$\mathcal A$ (and analogously for $\sigma_{\mathcal B}(\psi(a))$).
However, this is known to be impossible. In fact,
$\mathcal B$ is  a radical algebra as a subalgebra of 
the radical algebra $\mathfrak A_{X}$. On the other hand, non-zero elements 
$a \in \mathcal A$ are never quasi-nilpotent in $\mathcal A$, as 
$\sigma_{\mathcal A}(a)$ contains the spectrum $\sigma(a)$ of $a$ in 
$c_0$. Thus $\mathcal A = \{0\}$, and the above claim follows.
\end{proof}

For similar reasons the restriction of $\psi$ to the closed ideal
\[
\mathcal A_D = \{(x_n) \in c_0: x_n = 0 \ \textrm{for}\ n \notin D\}
\]
of $c_0$ is not an algebra homomorphism for any subset $D \subset \mathbb N$.
Nevertheless, we point out that 
the linear embedding $\psi$ from Proposition \ref{sum2} 
generates a commutative subalgebra in the quotient algebra.

\begin{prop}\label{prod1} 
Let $\psi: c_0 \to \A_X$ be the linear embedding constructed in  Proposition \ref{sum2},
and let $\mathcal B$ be the closed subalgebra of $\A_X$
generated by the image  $\psi(c_0)$.
Then $\mathcal B$ is a commutative algebra. 
\end{prop}

\begin{proof}
Note first that 
\begin{equation}\label{comm}
\psi(a) \psi(b) = \psi(b) \psi(a) \  \textrm{ for }\  a, b \in c_0,
\end{equation}
so that $\psi(c_0)$ is a commuting subset of $\A_X$.
Namely, if $a= (a_k), b = (b_k) \in c_0$, then
\[
\big(\sum_{k=1}^m a_kU_k\big) \big(\sum_{r=1}^n b_rU_r\big) =
\sum_{k=1}^{m \land n} a_kb_k U_k^2 = 
\big(\sum_{r=1}^n b_rU_r\big) \big(\sum_{k=1}^m a_kU_k\big),
\]
since by the construction in Proposition \ref{sum2} one has $U_kU_r = 0$ for $k \neq r$. 
Above $m \land n$ denotes the smaller of $m$ and $n$.
Hence  \eqref{comm} follows by passing to the limit in the quotient norm of $\A_X$. 

Recall next that the  finite linear combinations of the products 
$\psi(x_{k_{1}})^{m_1} \cdot \ldots \cdot \psi(x_{k_{s}})^{m_s}$,
where $x_{k_{1}}, \ldots, x_{k_{s}} \in c_0$ and $m_1,\ldots , m_s \ge 1$,
form a dense subset of $\mathcal B$. Since any two such finite linear combinations 
commute by \eqref{comm}, it follows by the standard approximation argument that 
$\mathcal B \subset \A_X$ is a closed commutative subalgebra.
\end{proof}

It should be kept in mind that the quotient algebras 
$\A_X$ are usually non-commutative. 
For completeness we next recall two simple conditions towards this.
 
\begin{ex}\label{nc}
Let $X$ and $Y$ be Banach spaces.\\
(i) Suppose that  $X \oplus X \approx X$ and that there are $A, B \in \mathcal K(X)$ so that 
$AB \notin \mathcal A(X)$ or $BA \notin \mathcal A(X)$. Then the compact operators 
$\widehat{A}$ and $\widehat{B}$ on $X \oplus X$ do not 
commute modulo $\mathcal A(X \oplus X)$, where 
\[
\widehat{A}(x,y) = (Ay,0), \ \widehat{B}(x,y) = (0,Bx), \quad (x,y) \in X \oplus X.
\]
(ii) Suppose that there are 
$U \in \mathcal K(Y,X)$ and $V \in \mathcal K(Y)$ so that $UV \notin \mathcal K(Y,X)$. Then
the operators $\widehat{U}, \widehat{V} \in \mathcal K(X\oplus Y)$ 
do not commute modulo $\mathcal A(X \oplus Y)$, where
\[
\widehat{U}(x,y) = (Uy,0), \ \widehat{V}(x,y) = (0,Vy), \quad (x,y) \in X \oplus Y.
\]
\end{ex}

\section{$\mathfrak A_Z$ for the Willis spaces $Z$}\label{Willis}

The main result  of this section demonstrates that $c_0$
embeds isomorphically into the quotient algebra 
$\mathfrak A_Z = \mathcal{K}(Z)/\mathcal{A}(Z)$, 
whenever $Z$ belongs to a class of Banach spaces constructed by Willis \cite{W}.

Our starting point is the following observation, 
which points out classes of Banach spaces $X$ for which
$\mathfrak A_{X}$ is always non-trivial. Part (ii) of the following proposition
confirms the expectation from \cite{D13} that  $\mathfrak A_Z$ is infinite-dimensional
for the spaces $Z$ from \cite{W}. Remark \ref{capex}.(ii) below indicates a  
different proof of this result suggested by Dales \cite{D13}.
Recall that the Banach algebra $\mathcal A$ is nilpotent,
 if there is $m \in \mathbb N$ 
such that the product $x_1 \cdot \ldots \cdot x_m = 0$ for any
$x_1, \ldots, x_m \in \mathcal A$.

\begin{prop}\label{bcapfact1}
(i) If the Banach space $X$ has the C.A.P., but fails to have the A.P.,
then  $\mathfrak A_X \neq \{0\}$.

(ii) If the Banach space $X$ has the B.C.A.P., but fails to have the A.P.,
then  $\A_{X}$ is infinite-dimensional.
More precisely, there is $V \in \mathcal K(X)$ such that $V^n \notin \mathcal A(X)$
for every $n \in \mathbb N$, and
the closed commutative subalgebra $\mathcal B$ 
generated by $V + \mathcal A(X)$ in $\A_{X}$ is infinite-dimensional.
\end{prop}

 \begin{proof} 
 The proof of (i) is the case $n = 1$ of the first part of the argument for 
 part (ii), so we will only verify (ii).
 
Let $n\in\N$ be arbitrary and fix $M \geq 1$ 
such that $X$ has the B.C.A.P. with constant $M$.
Since $X$ does not have the A.P.  by assumption, there is a compact subset 
 $K\subset X$ and a constant $c>0$ so that 
 \[
  \sup_{x\in K}||x-Vx||\geq c
\]
for every $V\in \mathcal F(X)$. On the other hand, 
since $X$ has the B.C.A.P. with constant $M$, 
there is a compact operator $U\in\mathcal K(X)$,
so that $||U||\leq M$ and
\[
\sup_{x\in K}||x-Ux||<\frac{c}{2M_{0}},
\]
where $M_{0}=\sum_{k=0}^{n-1} M^k$. 
Hence we get that
 \begin{align*}
||x-U^nx|| & = ||(I_X + U + \ldots + U^{n-1})(x-Ux)||
\leq \big(\sum_{k=0}^{n-1} \Vert U \Vert^k\big) \cdot ||x-Ux|| \\
& \le \big(\sum_{k=0}^{n-1} M^k\big) \cdot \frac{c}{2M_{0}} = c/2
\end{align*}
for every $x\in K$. It follows that $U^n\in\mathcal K(X)\setminus\mathcal A(X)$. 
Note that here the operator $U$ depends on $n$.

The preceding fact means that  $\A_{X}$  is radical Banach algebra which is not nilpotent.
It follows from general results about the Jacobson radical, 
see \cite[Prop. 1.5.6.(iv)]{D00}, 
that $\A_{X}$  is infinite-dimensional.
Moreover, by an application of the Baire theorem, see \cite{G} or \cite[Prop. 4.4.11.(b)]{P},
there is a compact operator
$V \in \mathcal K(X)$ such that $V^n \notin \mathcal A(X)$
for every $n \in \mathbb N$. It follows as above that the closed 
commutative subalgebra $\mathcal B$ generated by 
 $V + \mathcal A(X)$ in $\A_{X}$ is infinite-dimensional.
\end{proof}

\begin{remark}\label{capex}
(i) The first examples of Banach spaces, which have the metric C.A.P. but fail the A.P.,
were found by Willis \cite{W}. In Proposition \ref{bcapfact1} 
the classes of Banach spaces satisfying condition (i), respectively (ii), are really different. 
This fact is based on the spaces from \cite{W} in combination with known constructions
and an observation from \cite[Prop. 8.2]{C} 
or \cite[Cor. 2.4]{O}.

For this let $Y$ be a Banach space with the A.P. which fails to have the B.A.P., see  
e.g. \cite[1.e.18-1.e.20]{LT1}.
Observe that $Y$ also fails the B.C.A.P., since $\mathcal K(Y) = \mathcal A(Y)$. 
Consider  $X = Y \oplus Z$,  where $Z$ is a Banach space which has the 
 metric C.A.P. but fails to have the A.P., as given by \cite{W}.
In this event  $X = Y \oplus Z$ has the C.A.P., while $X$ fails both the A.P. and 
the B.C.A.P., as such properties are inherited by complemented subspaces.

(ii) Suppose that $X$ has the B.C.A.P. It is known as a consequence of the 
Cohen-Hewitt factorisation theorem, see e.g. \cite[2.9.26 and 2.9.37]{D00},
that any operator $S\in \mathcal K(X)$ factors as  $S = UV$ 
for suitable $U, V \in \mathcal K(X)$. Clearly this property passes to the quotient algebra
$\A_X$, so that $\A_X$  cannot be a nilpotent algebra if $\A_X \neq \{0\}$.
This also implies that $\A_X$ is infinite-dimensional
in part (ii) of Proposition \ref{bcapfact1},
as pointed out  by Dales \cite{D13}, 
but the argument included above uses only elementary facts.
\end{remark}

The argument for Proposition \ref{bcapfact1}.(ii) is algebraic in nature, 
and does not provide explicit information about $\A_X$.
The main purpose of this section is to sharpen the result
for the class of Banach spaces found by Willis \cite{W}.
More precisely, for any Banach space $X$ that fails the A.P., Willis obtains
a Banach space $Z_X$ that has the metric C.A.P., but fails to have the A.P. 
We begin by recalling the relevant details of the construction for our purposes.

Suppose that the Banach space $X$ fails to have the A.P., and  
fix a compact subset $K \subset X$ and a constant $c > 0$ so that 
 \begin{equation}\label{fail}
 \sup_{x \in K} \Vert x - Vx\Vert \ge c
 \end{equation}
for any finite rank operator $V \in \mathcal{F}(X)$. By a classical fact \cite[Prop. 1.e.2]{LT1} 
 one may assume that 
 $K = \overline{conv}\{\pm x_n: n \in \mathbb{N}\}$, where
 $0 <  \Vert x_n\Vert \le 1$ for all $n \in \mathbb{N}$ and $\Vert x_n\Vert \to 0$ as
 $n \to \infty$. For any  $0 < t <  1$  consider the closed compact symmetric subset  
 \[
 U_t = \overline{conv}\{\pm \frac{x_n}{\Vert x_n\Vert^t}: n \in \mathbb{N}\} 
 \]
of $X$  and let $Y_t \subset X$ be the linear span of $U_t$ normed by the Minkowski functional 
\[
\vert x \vert_t = \inf\{\lambda > 0: x \in \lambda U_t\}, \quad x \in Y_t.
\]
We will require the following facts: 
$(Y_t,\vert \cdot \vert_t)$ is a Banach space for any $0 < t < 1$, 
$U_t$ is the closed unit ball of $Y_t$ and
\begin{equation}\label{Wil1}
\vert x_n\vert_t\leq ||x_n||^t,  \quad n \in \mathbb{N}.
\end{equation}
Moreover, for any $0 < s < t < 1$ one has $U_s\subset U_t\subset B_X$ and 
$Y_s\subset Y_t\subset X$, so that
\begin{equation}\label{Wil2}
\Vert x\Vert \leq \vert x\vert_t \leq \vert x\vert_s, \quad x\in Y_s.
\end{equation}
In particular, the inclusions 
$Y_s\to Y_t$  and $Y_t \to X$ are bounded whenever $0 < s < t < 1$.

Let  $\mathcal{Z}$ be the linear span of the simple functions 
$\{y\chi_{(s,t)}: 0 < s < t < 1, y \in Y_s\}$ on $(0,1)$
equipped with the norm
 \[
 \Vert f\Vert = \int_0^1 \vert f(r)\vert_r dr = \sum_{k=1}^n \int_{u_k}^{v_k} \vert y_k\vert_r dr, 
 \]
for $f=\sum_{k=1}^n y_k\chi_{(u_k,v_k)}\in \mathcal Z$,
where $0 < u_1 < v_1 \le u_2 < \ldots  \le u_k  < v_k < 1$ and  
$y_k \in Y_{u_k}$ for  $k = 1, \ldots, n$. Note that  $f(r) \in Y_r$ for $f \in \mathcal{Z}$
and $0 < r < 1$, 
and that the above integral exists in the Riemann sense
by the monotonicity of the map $r \mapsto \vert y_k\vert_r$ for each $k$. 
Finally, let $Z = Z_X$ be the completion of $(\mathcal{Z},\Vert \cdot \Vert)$. 

Willis showed \cite[Prop. 1 and 2]{W} 
 that $Z$ has the metric C.A.P., but fails the A.P.
Note that $Z$ is non-reflexive and $Z^*$ is non-separable, since
 the closed subspace $L^1(0,1)x_1 \subset Z$ is isomorphic to $L^1(0,1)$, whence
$\mathcal{A}(Z)$ is also non-separable.
The following result is the main one of this paper. 
 
 \begin{theorem}\label{Willis1}
Suppose that $X$ fails to have the A.P. and let $Z = Z_X$ be the above Willis space 
associated to $X$.
Then there is a linear isomorphic embedding $\psi: c_0 \to \mathfrak A_Z$. 
 \end{theorem}
 
Recall from  Proposition \ref{prod} that the linear
embedding  $\psi: c_0 \to \mathfrak A_Z$ is not an algebra homomorphism. 
 
 \begin{proof}
In the argument we will verify in the following steps that the conditions (i)-(iv)
of Proposition \ref{sum2} are satisfied with $p = 1$. 
 
 \smallskip
 
 \noindent \textit{Step 1.} 
 Let $0 < s<t<1$. We claim that the set $U_s$ is compact in $Y_t$, 
 so that the inclusion operator $J_{s,t}: Y_s\to Y_t$ is compact.

Indeed, towards this note that  
\[
\vert \frac{x_n}{\Vert x_n\Vert^s}\vert_t\leq \Vert x_n\Vert^{t-s}\to 0
\]
as $n \to \infty$. This yields that the 
$\vert \cdot \vert_t$-closure
$\overline{conv}^{\vert \cdot\vert_t}\{\pm x_n/||x_n||^s:  n\in\N\}$ 
is compact in $Y_t$ by Mazur's theorem, so that  $J_{s,t}: Y_s\to Y_t$ is a
compact inclusion operator, because the unit ball $B_{Y_{s}} = U_s =
\overline{conv}\{\pm x_n/||x_n||^s:  n\in\N\}$.

\smallskip

 \noindent \textit{Step 2.} \textit{Claim:} the inclusion map 
$R_s: Y_s\to X$ is not an approximable operator for any $0 < s < 1$, and
 $\mathcal{A}(Y_s,Y_t) \varsubsetneq \mathcal{K}(Y_s,Y_t)$ whenever $0 < s < t < 1$.
 In particular,  $Y_t$ does not have the A.P. 
for any $0 < t < 1$. 

In fact, suppose to the contrary that
 $R_s \in \mathcal{A}(Y_s,X)$ and let $\varepsilon>0$ be given. 
 In view of the counter assumption 
 there exists a finite rank operator 
 \[
 R= \sum_{k=1}^n f_k\otimes y_k\in Y_s^*\otimes X
 \]
 such that 
 $\Vert R_s-R\Vert <\varepsilon/2$, where 
 $f_1, \ldots , f_n \in Y_s^*$ and $y_1, \ldots, y_n \in X$. In particular,
 note  that  $\Vert Rx-x \Vert <\varepsilon/2$ for $x\in K$, since 
 $K\subset U_s$. Note further that $K \subset Y_s$ is a relatively compact subset, 
 since  \eqref{Wil1} implies that 
 $\vert x_n\vert_s \le \Vert x_n\Vert^s \to 0$ as $n\to \infty$. 

Put $\delta=\varepsilon/(2\sum_{k=1}^n \Vert y_k\Vert) > 0$. According to  the argument 
in \cite[p. 33]{LT1} 
one knows that the range $R_s^*(X^*)$ is dense in $Y_s^*$ 
with respect to the topology of uniform convergence 
on the compact subsets of $Y_s$.
Hence we may pick functionals $g_1, g_2,\ldots,g_n \in X^*$ such that
\[
\sup_{x\in K}|\langle R_sx,g_k\rangle-\langle x,f_k\rangle| = 
\sup_{x\in K}|\langle x,R^*_sg_k\rangle-\langle x,f_k\rangle| < \delta
\]
for all $k=1,2,\ldots, n$.  Then the finite rank operator
$S = \sum_{k=1}^n g_k\otimes y_k\in \mathcal{F}(X)$
satisfies 
\begin{align*}
\Vert Sx-x\Vert & \leq \Vert Sx-Rx \Vert + \Vert Rx-x\Vert \\
& \le \sum_{k=1}^n \vert \langle R_sx,g_k\rangle - \langle x,f_k\rangle \vert \cdot \Vert y_k\Vert 
+ \Vert Rx-x\Vert 
<\varepsilon/2+\varepsilon/2=\varepsilon
\end{align*}
for  $x\in K$. This contradicts property \eqref{fail} of the compact set $K \subset X$ 
once  $0 < \varepsilon < c$. 

Moreover, since $R_s = R_t \circ J_{s,t}$, where $R_s \notin \mathcal{A}(Y_s,X)$,  
it follows from the result in Step 1 that 
$J_{s,t} \in \mathcal{K}(Y_s,Y_t) \setminus \mathcal{A}(Y_s,Y_t)$.
 
 \smallskip
 
 \noindent \textit{Step 3.}  For any fixed $s\in (0,1)$ define the linear map $T_s: Y_s\to Z$ by 
\[
T_s(y)=\frac{1}{1-s}y\chi_{(s,1)}, \quad y \in Y_s.
\] 
We claim that  $T_s$ is a compact operator. (This fact was used in \cite{W} for $s = 1/2$.)

To verify the claim, note that for any $n\in\N$ one has
\begin{align*}
\Vert T_s(x_n) \Vert & = \frac{1}{1-s} \int_0^1 \vert x_n\chi_{(s,1)}(r)\vert_r dr =
\frac{1}{1-s}\int_s^1 \vert x_n \vert _r dr\\
& \leq \frac{1}{1-s}\int_s^1 \Vert x_n\Vert^r dr =
\frac{1}{(1-s)}\frac{(\Vert x_n\Vert - \Vert x_n\Vert^s)}{\ln \Vert x_n\Vert }\\
& \le \frac{1}{(1-s)}\frac{\Vert x_n\Vert^s}{\vert \ln \Vert x_n\Vert  \vert},
\end{align*}
so that 
\[
\Vert T_s(x_n/\Vert x_n\Vert^s) \Vert  \le \frac{1}{(1-s)} \frac{1}{\vert \ln \Vert x_n\Vert  \vert}.
\]
Since $x_n \to 0$ in $X$ as $n \to \infty$, we conclude that
$\Vert T_s(x_n/\Vert x_n\Vert^s)\Vert \to 0$ as $n \to \infty$. By Mazur's theorem
the set 
\[
 \overline{conv}\{\pm T_s(x_n/\Vert x_n\Vert^s): n\in\N\}
\]
is compact in $Z$, and 
since $T_sU_s\subset \overline{conv}\{\pm T_s(x_n/\Vert x_n\Vert^s):  n\in\N\}$,
we get that $T_s$ is a compact operator $Y_s \to Z$.

\smallskip

For the remainder of the argument we next fix intertwining sequences 
$(s_n)$ and $(t_n)$ such that 
$0 < s_1 < t_1 < s_2 < t_2 < \ldots < 1$,
 where $s_n \to 1$  as $n \to \infty$. Then 
 $f \mapsto P_nf = f\chi_{(s_n,t_n)}$ is a well-defined  norm-$1$ projection of 
 $Z$  onto the closed subspace $Z_n = P_nZ \subset Z$ for any $n \in \N$, since 
 \[
 ||P_n f||=   \int_{s_n}^{t_n}\vert f(r)\vert_r dr\leq \int_0^1 \vert f(r) \vert_r dr= \Vert f \Vert
 \]
for any $f \in Z$.
Moreover, we may define a bounded linear operator  $J: Z\to X$ by 
\begin{equation}\label{map1}
J(f)=\int_0^1 f(r) dr, \quad  f \in Z.
\end{equation}
In fact, for a simple function $f=\sum_{k=1}^n y_k\chi_{(u_k,v_k)}\in \mathcal Z$,
where $0 < u_1 < v_1 \le u_2 < \ldots  \le u_k  < v_k < 1$ and  
$y_k \in Y_{u_k}$ for each $k = 1, \ldots, n$, put
\[
J(f)=\sum_{k=1}^n (v_k-u_k)y_k = \int_0^1 f(r) dr. 
\]
Observe that by \eqref{Wil2} one has 
\[
\Vert J(f)\Vert \leq \sum_{k=1}^n (v_k-u_k)\Vert y_k\Vert \le
\sum_{k=1}^n \int_{u_{k}}^{v_k} \vert y_k\vert_r dr = \Vert f\Vert
\]
for such simple functions $f\in \mathcal Z$, so that the above map admits by density a  
bounded linear extension $J$ to $Z$  that satisfies (\ref{map1}).

\smallskip

The following step verifies the crucial condition (ii) from Proposition \ref{sum2}.

 \smallskip

\noindent \textit{Step 4.} \textit{Claim:} 
$\mathfrak A_{Z_n}  \neq\{0\}$ for all $n \in \mathbb{N}$. 

Towards this claim let $J_n:Z_n\to Z$ denote the inclusion map for $n\in\N$. 
It is easy to check that the inclusion $R_{s_{n}}: Y_{s_n}\to X$ factors as 
\begin{equation}\label{factor1}
R_{s_{n}} = c_nJ J_n P_n T_{s_{n}}, 
\end{equation}
 where $c_n =\frac{1-s_n}{t_n-s_n}$ and the operators 
 $T_{s_n}:Y_{s_n}\to Z$ and $J: Z\to X$ are defined in Step 3, respectively in (\ref{map1}). 
Since $R_{s_{n}}$ is not an approximable operator by Step 2, it follows 
from \eqref{factor1} that 
$P_n T_{s_n}$ is not approximable $Y_{s_n}\to Z_n$. 
Recall from Step 3 that $T_{s_n}$ is a compact operator, whence 
\[
P_nT_{s_n}\in \mathcal K(Y_{s_n},Z_n)\setminus \mathcal A(Y_{s_n},Z_n). 
\]
This means  that $Z_n$ does not have the A.P. 
On the other hand,  $Z_n$ is a $1$-complemented subspace of $Z$, 
where $Z$ has the metric C.A.P. by \cite[Prop. 2]{W}, 
so  $Z_n$ also has the  metric C.A.P. 
Consequently $\mathfrak A_{Z_n}\neq\{0\}$ in view of  Proposition \ref{bcapfact1}.

 \smallskip

We finally check the remaining conditions of Proposition \ref{sum2} with 
$p = 1$ for $Z$ and the sequence $(P_n) \subset \mathcal L(Z)$ of norm-$1$ projections 
onto the subspaces $Z_n$. Condition (iii) 
is obvious for $p = 1$. Moreover,
in view of the disjoint supports on $(0,1)$ of the functions in the subspaces $Z_n$, 
the integration norm in $Z$ satisfies
\begin{equation}\label{dis1}
\sum_{k=1}^\infty \Vert P_kf\Vert =   \sum_{k=1}^\infty  \int_{s_{k}}^{t_k} \vert f(r)\vert_r dr \le 
\int_0^1  \vert f(r)\vert_r dr = \Vert f\Vert
\end{equation}
for any $f \in Z$. Hence condition (iv) is also satisfied, and the proof of the theorem is 
completed by an application of  Proposition \ref{sum2}.
\end{proof}

\begin{remark}
The proof   in \cite[Prop. 2]{W}  that $Z$ has the metric C.A.P.  uses a sequence 
$(T_n) \subset \mathcal K(Z)$ of vector-valued convolution operators on $Z$.
It is conceivable that  basic subsequence techniques applied to 
 $(T_{n} + \mathcal A(Z))$ in $\A_Z$ might produce isomorphic copies of $c_0$,  
or even of other spaces.
However, for this approach to be feasible one is likely to need 
descriptions of the dual spaces $Z^*$, $Z^{**}$ and $\mathcal K(Z)^*$.
\end{remark}

Recall from the proof of Corollary \ref{corsum3} that there are closed linear subspaces  
$X \subset \ell^p$ that fail the A.P. for 
any $1 < p < \infty$ and $p \neq 2$. 
Willis \cite[Prop. 3 and 4]{W} includes 
a modified construction, where these subspaces $X$
lead to a separable reflexive
space $Z^{\sharp}_p$, so that again $Z^{\sharp}_p$ has the metric C.A.P. 
but fails to have the A.P. 
The space $Z^{\sharp}_p$ is a quotient of a closed subspace of $L^p(0,1)$,
see \cite[Prop. 3]{W}.
We observe next  for completeness that Theorem \ref{Willis1} 
remains valid for $Z^{\sharp}_p$, 
but we will not reproduce all the technical details. 

\begin{ex}\label{Willisv2}
There is a linear isomorphic embedding 
$\psi: c_0 \to \mathfrak A_{Z^{\sharp}_p}$ for  $1 < p < \infty$ and $p \neq 2$. 
\end{ex}

\begin{proof} 
The construction in \cite[p. 103]{W} is based on certain modified 
compact convex symmetric sets $V_t \subset X$ for $0 < t < 1$ that also depend on $p$. 
As before one introduces associated Banach spaces $(W_t,\vert \cdot \vert_t)$ 
such that $W_s \subset W_t \subset X$ for $0 < s < t < 1$, where $W_t$ is the linear span
of $V_t$ in $X$. The space $Z^{\sharp}_p$ is the completion of
the linear span of  the simple functions 
$\{y\chi_{(s,t)}: 0 < s < t < 1, y \in W_s\}$ on $(0,1)$
equipped with the norm
 \[
 \Vert f\Vert_p = \Big( \int_0^1 \vert f(r)\vert^p_r dr\Big)^{1/p}.
 \]
 Fix  once more sequences $(s_n)$ and $(t_n)$ such that 
$0 < s_1 < t_1 < s_2 < t_2 < \ldots < 1$,
 where $s_n \to 1$  as $n \to \infty$, and let $P_n$ be the  
norm-$1$ projection of 
 $Z^{\sharp}_p$  onto the closed subspace 
 $U_n := P_n(Z^{\sharp}_p) \subset Z^{\sharp}_p$
 consisting of the functions supported on $(s_n,t_n)$ for each $n$. 
 
 Let $R_s: Y_s \to X$ be the inclusion map from Theorem \ref{Willis1}, and let 
  $T^{\sharp}_{s_n}: Y_{s_n}\to Z^{\sharp}_p$ and $J^{\sharp}: Z^{\sharp}_p\to X$ 
  be analogous operators
 to those of Step 3 of that argument. One verifies as in Step 3 that 
  $T^{\sharp}_{s_n} \in \mathcal K(Y_{s_n},Z^{\sharp}_p)$ for each $n$.
  Moreover, 
  \[
 R_{s_{n}} = c_nJ^{\sharp} J_n P_n T^{\sharp}_{s_{n}}, 
\]
 where $c_n =\frac{1-s_n}{t_n-s_n}$. Thus
 $P_n T^{\sharp}_{s_{n}}$ is a compact non-approximable operator
 $Y_{s_n} \to  U_n$, since $R_{s_{n}}$ is non-approximable by Step 2 of 
 Theorem \ref{Willis1}.
 Hence $U_n$ has the metric C.A.P. but fails the A.P. for all $n \in \mathbb N$, 
 so that $\mathfrak A_{U_n}  \neq\{0\}$ by Proposition \ref{bcapfact1}.
 
 The rest of the argument is similar to  that of Theorem \ref{Willis1}, since
conditions (iii) and (iv) from Proposition \ref{sum2} are satisfied for $Z^{\sharp}_p$
 with the exponent $p \in (1,\infty)$. 
 Namely, by the disjointness of the supports one gets that 
 \[
 \Vert \sum_{n=1}^\infty f_n\Vert_p = \big(\sum_{n=1}^\infty \Vert f_n\Vert_p^p\big)^{1/p}
 \]
 whenever $f_n \in P_n(Z^{\sharp}_p)$ 
 for $n \in \mathbb N$ and $\sum_{n=1}^\infty \Vert f_n\Vert^p$ is finite.
\end{proof}
 
We note for completeness  that various other results and applications 
which involve the Willis spaces $Z_X$ can be found e.g. in 
\cite{T}, \cite{CJ}, \cite {LO} and \cite{OZ}.

\section{Further examples on universal spaces}\label{further}

In this section we discuss explicit examples of universal  Banach spaces $X$, 
where the quotient algebra $\A_X$ is large. The examples include 
the complementably universal 
conjugate  space $C_1^*$ of Johnson, and some universal compact 
factorisation spaces that originate in the work of Johnson and Figiel.

We first recall from  \cite[1.e.7]{LT1} that $X$ has the A.P. whenever the dual space 
$X^*$ has the A.P., and that the converse fails in general.
We next show that analogous facts hold for the respective quotient algebras.

 \begin{prop}\label{dual}
 Let $X$ be any Banach space and define $\theta: \A_X \to \A_{X^*}$ by 
 \[
 \theta(S + \mathcal{A}(X)) = S^* + \mathcal{A}(X^*), \quad S \in \mathcal K(X).
 \]
 Then $\theta$ is an isometric linear embedding  $\A_X \to \A_{X^*}$,
 which is an anti-homomorphism.
 \end{prop}
 
The fact that $\theta$ is an anti-homomorphism 
means that $\theta$ reverses products, that is,
 \[
 \theta (ST+\mathcal{A}(X)) = \theta(T+\mathcal{A}(X)) \cdot \theta(S+ \mathcal{A}(X))
 \]
for $S, T \in \mathcal K(X)$. This property is obvious from $(ST)^* = T^*S^*$.

\begin{proof}
Note first that $U \in  \mathcal{A}(X)$ implies that $U^* \in \mathcal{A}(X^*)$,
so $\theta$ is a well-defined map. Moreover, $\theta$ is an isometry 
since the principle of local reflexivity implies that 
\[
dist(S^*, \mathcal{A}(X^*)) = dist(S, \mathcal{A}(X))
\]
for any $X$ and any compact operator $S \in \mathcal K(X)$,
see  e.g. \cite[Prop. 2.5.2]{CS}.
\end{proof}

\begin{remark}
Let $X$ be any Banach space. The map $S \mapsto S^*$ also induces well-defined
quotient maps $\psi: \mathcal L(X)/\mathcal A(X) \to  \mathcal L(X^*)/\mathcal A(X^*)$
 and  $\chi:  \mathcal L(X)/\mathcal K(X) \to  \mathcal L(X^*)/\mathcal K(X^*)$
as above, but $\psi$ and $\chi$ 
 behave differently from $\theta$.
 
Namely, by the principle of local reflexivity one also has 
 \begin{equation*}
dist(S, \mathcal{A}(X)) \le 5 \cdot dist(S^*, \mathcal{A}(X^*)) = 
5 \cdot \Vert \psi(S+\mathcal{A}(X))\Vert
\end{equation*}
for any  bounded operator $S \in \mathcal L(X)$ and any Banach space $X$,
see e.g. \cite[Prop. 2.5.4]{CS}. Moreover, $\psi$ is not always an isometry, 
see (2.5.7) and (2.5.9) in \cite{CS}.

Finally, by \cite[Example 2.5]{T} there is a Banach space $X$ 
such that  $\chi$ is not even bounded from below
$\mathcal L(X)/\mathcal K(X) \to \mathcal L(X^*)/\mathcal K(X^*)$
in the respective quotient norms.
Above $X$ does not have the B.C.A.P. by \cite[Thm. 2.4]{T}.
\end{remark}

We next include examples where the difference between 
$\A_X$ and $\A_{X^*}$ is large. 
Towards this  fix $1 \le p \le \infty$ and let
$C_p  = (\oplus_{m\in \mathbb N} G_m)_{\ell^p}$ be the class of universal spaces
constructed by Johnson \cite{J71}. For the definition  fix a sequence 
$(E_n)_{n\in \mathbb N}$ of finite-dimensional spaces 
which is dense in the Banach-Mazur distance $d_{BM}$ in the class of 
all finite-dimensional spaces, 
and repeat each $E_n$ infinitely often to obtain the listing 
$(G_m)_{m \in \mathbb N}$. Above
\[
d_{BM}(E,F) = \inf \{ \Vert T \Vert \cdot \Vert T^{-1}\Vert: T \ \textrm{linear isomorphism}\ E \to F\}.
\]
For unity of notation $p = \infty$ 
corresponds again to a direct $c_0$-sum.
Note that with this convention  $C_p$ has the A.P. for any 
$1 \le p \le \infty$.

 \begin{ex}\label{john}
The quotient $\A_{C_1^*}$ contains a linear isomorphic copy of $c_0$,
 but  $\A_{C_1} = \{0\}$.
 \end{ex}
 
\begin{proof}
By Corollary \ref{corsum3} there is a separable reflexive 
subspace $X \subset \ell^p$ such that $\A_{X}$ 
contains a linear isomorphic copy of $c_0$. 
Proposition \ref{dual} implies that $\A_{X^*}$ contains an isomorphic  copy of $c_0$. 
(Since $X$ is reflexive, the latter conclusion is also 
seen directly by passing to adjoints. Obviously
$X = Z^{\sharp}_p$ from Example \ref{Willisv2} can also be used for the above purpose,
but Theorem \ref{Willis1} and
Example \ref{Willisv2} depend on longer arguments.)

Johnson \cite[Thm. 1]{J72} showed that  $C_1^*$ is 
complementably universal for separable spaces. 
This result says that
$X^*$ is isometric to a $1$-complemented subspace of $C_1^*$, that is, 
\[
C_1^* = W \oplus M,
\]
where $W$ is isometric to $X^*$ and $W$ is complemented in $C_1^*$ 
by a norm-$1$  projection.
Finally, by applying the subsequent Proposition \ref{dsum} 
to the direct sum $C_1^* = W \oplus M$, 
we obtain that $\A_{C_1^*}$  contains a linear isomorphic copy of  $c_0$.
 \end{proof}
 
The following observation  can be viewed as a more precise version of 
Proposition \ref{sum2} for finite direct sums, and this fact will also be used later on.
It will be convenient to denote operators
$S \in \mathcal L(X\oplus Y)$ on the direct sum $X \oplus Y$ as 
$2\times 2$-operator matrices
\[
S = \begin{pmatrix} S_{11} & S_{12} \\ S_{21} & S_{22}\end{pmatrix}
\]
in the canonical way. This means that $S_{11} = P_1SJ_1$, $S_{12} = P_1SJ_2$,
$S_{21} = P_2SJ_1$ and $S_{22} = P_2SJ_2$,
where $P_1$ is the projection $X\oplus Y \to X$, $J_1$ is the inclusion $X \to X\oplus Y$,
and $P_2$ and $J_2$ are the analogous operators for the summand $Y$. 

\begin{prop}\label{dsum}
Let $X$ and $Y$ be Banach spaces, where $\A_X \neq \{0\}$. 
Then $\A_X$  is algebra isomorphic to 
the complemented subalgebra 
\[
\mathcal M = \Big\{ \begin{pmatrix} S_{11} & 0 \\ 0 & 0\end{pmatrix}+\mathcal A(X\oplus Y):
S_{11} \in \mathcal K(X)\Big\}
\]
of $\A_{X\oplus Y}$. 
\end{prop}

\begin{proof}
We first verify that the map $\Phi$ given by 
\[
\Phi(S + \mathcal A(X\oplus Y)) =
\begin{pmatrix} S_{11} & 0 \\ 0 & 0\end{pmatrix}+\mathcal A(X\oplus Y), 
\quad S \in \mathcal K(X\oplus Y),
\]
defines a bounded linear projection $\A_{X\oplus Y} \to \mathcal M$. 
Towards this fact we consider the map 
$\Psi: \mathcal K(X\oplus Y) \to \mathcal K(X\oplus Y)$
defined by 
\[
\Psi(S) = \begin{pmatrix} S_{11} & 0 \\ 0 & 0\end{pmatrix}, \quad  
S = \begin{pmatrix} S_{11} & S_{12} \\ S_{21} & S_{22}\end{pmatrix} \in 
\mathcal K(X\oplus Y).
\] 
Observe that $\Psi$ is a
bounded linear operator,  $\Vert \Psi \Vert = 1$, since we may write 
\[
\Psi(S) = J_1(P_1SJ_1)P_1, \quad S \in \mathcal K(X\oplus Y).
\]
Clearly $\Psi(\mathcal A(X\oplus Y)) \subset \mathcal A(X\oplus Y)$, so that
$\Psi$ induces the above map $\Phi$ on the quotient space
$\A_{X \oplus Y}$, whence $\Vert \Phi \Vert \le \Vert \Psi\Vert = 1$. 

Next define $\alpha: \A_X \to \A_{X\oplus Y}$ by 
\[
\alpha (U + \mathcal A(X)) = \begin{pmatrix} U & 0 \\ 0 & 0\end{pmatrix}+\mathcal A(X\oplus Y),
\quad U \in \mathcal K(X).
\]
It is not difficult to check that $\alpha$ is a well-defined 
algebra homomorphism, and by arguing as above,
one gets that $\Vert \alpha \Vert = 1$. Moreover, if $U \in \mathcal K(X)$
and $A \in \mathcal A(X \oplus Y)$, then
\[
\Vert  \begin{pmatrix} U & 0 \\ 0 & 0\end{pmatrix} - A \Vert \ge
\Vert P_1\big(\begin{pmatrix} U & 0 \\ 0 & 0\end{pmatrix} - A\big)J_1\Vert = 
\Vert U - A_{11}\Vert \ge dist(U,\mathcal A(X)).
\]
We conclude that $\Vert \alpha(U  +\mathcal A(X))\Vert = \Vert U +\mathcal A(X)\Vert$
holds for all $U \in \mathcal K(X)$, so that
$\alpha$  defines an algebra isomorphism $\A_X \to \mathcal M$.
\end{proof}

The universal conjugate space $C_1^*$ is not separable, but 
there are spaces $Y$ with $Y^*$ separable, which display the
same duality behaviour as $X = C_1$ in Example \ref{john}. For this we 
use another known construction. 

\begin{ex}\label{JL}
 Let $X$ be a separable reflexive Banach space
for which $\A_{X^*}$ contains a linear isomorphic copy of $c_0$, as in Example \ref{john}. 
A  construction due to James and Lindenstrauss, see \cite[1.d.3]{LT1}, 
provides a Banach space $Z$ (which depends on $X$)
such that $Z^{**}$ has a Schauder basis and $Z^{**}/Z \approx X$. It follows that 
\[
Z^{***} \approx Z^* \oplus X^*, 
\] 
so that $Z^{***}$ is separable.  
Proposition \ref{dsum} implies that $\A_{Z^{***}}$ 
contains a linear isomorphic copy of $c_0$, 
while $\A_{Z^{**}} = \{0\}$. 

Actually, in this example one may explicitly identify $\A_{Z^{***}}$ with $\A_{X^*}$, since the  
component operators with respect to the decomposition $Z^{***} \approx Z^* \oplus X^*$
satisfy $\mathcal{K}(Z^*) = \mathcal{A}(Z^*)$, $\mathcal{K}(X^*,Z^*) = \mathcal{A}(X^*,Z^*)$
and $\mathcal{K}(Z^*,X^*) = \mathcal{A}(Z^*,X^*)$.
Here one uses the fact that $Z^{**}$ as well as $Z^*$  have the A.P.
\end{ex}

Johnson \cite[Thm. 1]{J71} and Figiel \cite[Prop. 3.1]{F} showed 
that for any $1 \le p \le \infty$
the spaces $C_p$ have the following compact factorisation property: 
given any Banach spaces $X, Y$ and compact operator $T \in \mathcal K(X,Y)$,
 there is a closed infinite-dimensional subspace $W \subset C_p$ as well as
 compact operators 
$A_0\in \mathcal K(X,W)$, $B_0 \in \mathcal K(W,Y)$ so that $T = B_0 A_0$.
Following Aron et al. \cite{ALRR} and \cite{MO} we consider the direct sum
\[
Z^p_{FJ} = \big( \oplus_{W} W)_{\ell^p},
\]
where $W$ runs through all the closed infinite-dimensional subspaces of 
$C_p$ in the summation. The
case $p = \infty$ is interpreted as a vector-valued $c_0$-type direct sum
(see below).

The space $Z^p_{FJ}$ is a Figiel-Johnson universal compact factorisation 
space for $1 \le p \le \infty$, since  $Z^p_{FJ}$ has the following universal property:
for any Banach spaces $X, Y$ and  $T \in \mathcal K(X,Y)$ there is 
$A \in \mathcal K(X,Z^p_{FJ})$ and $B\in \mathcal K(Z^p_{FJ},Y)$ so that $T = BA$.
In fact, if $T = B_0A_0$ factors compactly through some $W \subset C_p$ as above,
put $A = J_WA_0$ and $B = B_0P_W$, 
where $J_W: W \to Z^p_{FJ}$ is the inclusion map from  the $W$:th component, 
and $P_W: Z^p_{FJ} \to W$ is the corresponding canonical projection. 
Recall further, see \cite[p. 341]{J71}, that $C_p$ and $C_q$ are totally incomparable
spaces for $p \neq q$, that is, $C_p$ and $C_q$ do not contain any isomorphic
infinite-dimensional closed subspaces. Hence
$Z^p_{FJ}$ and $Z^q_{FJ}$ are not isomorphic  for $p \neq q$.

Let $\Gamma$ be an uncountable set. Recall that $(x_{\gamma}) \in c_0(\Gamma)$
if the set $\{\gamma \in \Gamma: \vert x_{\gamma}\vert > \alpha\}$ is finite for all
$\alpha > 0$. 
It is well known that 
$c_0(\Gamma)$ is a non-separable Banach space 
equipped with the supremum norm $\Vert \cdot \Vert_\infty$. 
We next show that the universal factorisation spaces $Z^p_{FJ}$ have a non-separable
quotient algebra $\A_{Z^p_{FJ}}$ for all $1 \le p \le \infty$.

\begin{theorem}\label{FJv1}
There is an uncountable set $\Gamma$ so that 
$c_0(\Gamma)$ embeds as a linear isomorphic subspace into 
$\A_{Z^p_{FJ}}$  for any $1 \le p \le \infty$.
\end{theorem}

\begin{proof}
We fix $p \in [1,\infty]$ for the duration of the argument. The main novelty of the argument 
is contained in the following claim.

\smallskip

\noindent \textit{Claim.} There is an uncountable family $\{Z_\gamma: \gamma \in \Gamma\}$
of distinct closed subspaces of $C_p$ such that $\A_{Z_{\gamma}} \neq \{0\}$
for all $\gamma \in \Gamma$. 

\smallskip

Observe first that there
is a closed subspace $Z \subset C_p$ such that $\A_Z \neq \{0\}$. 
In fact, consider Banach spaces $X$ and $Y$ such that 
$\mathcal A(X,Y) \varsubsetneq \mathcal K(X,Y)$.
The preceding Johnson-Figiel  factorisation  
applied to $A \in \mathcal K(X,Y)\setminus \mathcal A(X,Y)$ gives
a closed subspace $W \subset C_p$, as well as $B\in K(W,Y)$ and $C \in K(X,W)$,
so that $A=BC$.
Here $B$ cannot be an approximable operator,  since $A$ is not approximable.
Thus we may also factor $B\in K(W,Y)$ compactly through a closed subspace 
$V\subset C_p$.
This produces a compact non-approximable operator $S: W\to V$, so that 
$U = \begin{pmatrix} 0 & 0\\ S & 0\end{pmatrix} \notin \mathcal A(Z)$,
where $Z :=W\oplus V\subset C_p\oplus C_p\cong C_p$. 
Here $U(x,y) = (0,Sx)$ for $(x,y) \in Z$. Consequently  $\A_Z\neq\{0\}$.
 
We may thus fix $T \in \mathcal K(Z) \setminus \mathcal A(Z)$.
Next we find a sequence of distinct closed subspaces $Z_n$, where
$Z \subset Z_{n} \subset C_p$, as well as 
operators $T_n \in  \mathcal K(Z_n) \setminus \mathcal A(Z_n)$ for $n \in \mathbb N$.
In fact, since the subspace $Z \subset C_p$ does not have the A.P., the quotient space
 $C_p/Z$ is infinite-dimensional, and we may pick 
a normalised basic sequence $(x_n + Z)_{n \in \mathbb N}$ 
in  $C_p/Z$, see   e.g. \cite[Thm. 1.a.5]{LT1}.
Consider the closed linear subspace
 \[
 Z_n = Z + [x_k: 1 \le k \le n] \subset C_p, \quad n \in \mathbb N,
 \]
 where we use  
 $[D]$ to denote the closed linear span in $C_p$ of any subset $D \subset C_p$.
It is clear that $Z_n \neq Z_m$ whenever $n \neq m$,  
since the sequence $(x_k)$ is independent modulo $Z$. Moreover,
vectors  $x \in Z_n$ have a unique representation $x = z + \sum_{k=1}^n c_kx_k$,
 where $z \in Z$. Hence there is for any $n \in \mathbb N$  
 a bounded projection 
 $P_n: Z_n \to Z$, which is defined by $P_nx = z$ for 
 $ x = z + \sum_{k=1}^n c_kx_k \in Z_n$.
We next define the linear map $T_n$ in $Z_n$ by 
  \[
 T_nx = Tz + \sum_{k=1}^n c_kx_k, \quad x = z + \sum_{k=1}^n c_kx_k \in Z_n,
 \]
for  $n \in \mathbb N$.
One may identify 
$T_n = \begin{pmatrix} T & 0 \\ 0 & I_n\end{pmatrix}$ with respect to the 
direct sum decomposition
$Z_n = Z \oplus  [x_k: 1 \le k \le n]$, where
$I_n$ is the associated identity map.
It follows that $T_n: Z_n \to Z_n$ is a 
compact non-approximable operator,
since $T \in \mathcal K(Z) \setminus \mathcal A(Z)$. 
Hence  $\A_{Z_n} \neq \{0\}$ for each $n \in \mathbb N$. 
  
Recall from \cite[p. 341]{J71} that 
$C_p\cong \big(\bigoplus_{k=1}^\infty C_p\big)_{\ell^p}$,
so that the above construction can be applied coordinatewise in each summand.
This gives an uncountable  family  $\{Z_\gamma: \gamma \in \Gamma\}$ 
of closed subspaces of $C_p$, where 
\[
Z_\gamma = \big(\oplus_{j \in \mathbb N} (Z^{(j)})_{n_j} \big)_{\ell^p} \subset  
\big(\bigoplus_{k=1}^\infty C_p\big)_{\ell^p} \cong C_p
\]
and $\gamma = (n_j)_{j \in \mathbb N} \in \Gamma := \prod_{j \in \mathbb N} \mathbb N$.
Here $Z^{(j)}$ denotes an isomorphic  copy of $Z$ in the $j$:th summand $C_p$ 
in the above direct sum. 
To check that $Z_{\gamma} \neq Z_{\gamma'}$
whenever $\gamma \neq \gamma'$, note that if
$\gamma \neq \gamma'$, where $\gamma = (n_j)$ and  $\gamma' = (m_j)$, then
$n_j \neq m_j$ for some $j \in \mathbb N$. Thus  the respective $j$:th component spaces
satisfy $(Z^{(j)})_{n_j} \neq (Z^{(j)})_{m_j}$, whence $Z_{\gamma} \neq Z_{\gamma'}$.
Moreover,  $(Z^{(1)})_{n_1}\subset Z_\gamma$
is a complemented subspace,
so it follows from Proposition \ref{dsum}
that $\A_{Z_\gamma} \neq \{0\}$ for any $\gamma \in \Gamma$. 
This finishes the verification of the Claim.

To complete the argument we consider  the complemented subspace 
$Y = \big(\oplus_{\gamma \in \Gamma} Z_\gamma\big)_{\ell^p}$ of $Z^p_{FJ}$. 
By Proposition \ref{dsum} it suffices to find  a linear isomorphic 
embedding $c_0(\Gamma) \to \A_Y$.
This step is a modification of the corresponding argument in 
Proposition \ref{sum2}.
Firstly, for each $\gamma \in \Gamma$ pick  a compact operator 
$T_{\gamma} \in \mathcal K(Z_\gamma)$ so that 
\[
dist(T_{\gamma},\mathcal A(Z_\gamma)) =1, \quad \Vert T_{\gamma}\Vert < 2.
\]
Let $U_\gamma = J_\gamma T_{\gamma} P_\gamma$, where 
$J_\gamma: Z_\gamma \to Y$ is the natural inclusion map and 
$P_\gamma: Y \to Z_\gamma$ the natural projection for $\gamma \in \Gamma$. 
Then the series
$\sum_{\gamma \in \Gamma} a_\gamma U_\gamma$ defines a compact operator 
$Y \to Y$, and 
\begin{equation}\label{bound}
\Vert \sum_{\gamma \in \Gamma} a_\gamma U_\gamma \Vert \le
2 \cdot \sup_{\gamma} \vert a_\gamma \vert,
\end{equation}
for any $(a_\gamma) \in c_0(\Gamma)$ .
In fact, consider the finite sets $\Gamma_r = \{\gamma \in \Gamma: \vert a_{\gamma}\vert > 1/r\}$
and the associated compact operator
$V_r = \sum_{\gamma \in \Gamma_r}   a_\gamma U_\gamma$ 
on $Y$ for $r \in \mathbb N$. It follows that $(V_r) \subset \mathcal K(Y)$ is a
Cauchy-sequence, since
\[
\Vert V_{r+s} - V_r \Vert =  
\Vert \sum_{\gamma \in \Gamma_{r+s} \setminus \Gamma_r}   a_\gamma U_\gamma \Vert
\le 2/r
\]
for each $r, s \in \mathbb N$. Thus $\sum_{\gamma \in \Gamma} a_\gamma U_\gamma
= \lim_{r\to \infty} V_r$ defines a compact operator on $Y$.
To check \eqref{bound}, put 
$\Gamma_0 = \cup_{r \in \mathbb N} \Gamma_r$
and let $x = (x_\gamma) \in Y$.  It follows that 
\[
\Vert \sum_{\gamma \in \Gamma} a_\gamma U_\gamma x \Vert_{\ell^p} 
=  \Vert \sum_{\gamma \in \Gamma_0} a_\gamma J_\gamma T_\gamma x_\gamma \Vert_{\ell^p}
\le 2 \cdot \sup_{\gamma} \vert a_\gamma \vert \cdot \Vert x \Vert_{\ell^p}.
\]

Finally, suppose that $(a_\gamma) \in c_0(\Gamma)$ and 
$\vert a_\beta \vert = \Vert (a_\gamma)\Vert_\infty$ for some $\beta \in \Gamma$.
If $V \in \mathcal A(Y)$ is arbitrary, then
\[
\Vert \sum_{\gamma \in \Gamma} a_\gamma U_\gamma - V\Vert
\ge \Vert P_\beta \big(\sum_{\gamma \in \Gamma} a_\gamma U_\gamma - V)J_\beta\Vert
= \Vert a_\beta T_\beta - P_\beta V J_\beta\Vert \ge \vert a_\beta\vert,
\]
since $U_\gamma J_\beta = 0$ for $\gamma \neq \beta$ by construction.

This completes the proof  that $c_0(\Gamma)$ embeds into $\A_Y$, and 
consequently of the theorem, after an application of Proposition \ref{dsum}.
\end{proof}

\section{Concluding remarks}\label{conclude}

We remind that it seems to be unknown whether $\A_P \neq \{0\}$,
where $P$ belongs to the family of spaces  constructed by 
Pisier \cite{P1}, \cite[section 10]{P2}. The spaces  $P$ fail the A.P. 
and $\mathcal{A}(P) = \mathcal N(P)$, where $\mathcal N(P)$ 
denotes  the space of  nuclear operators $P \to P$. Moreover, 
$\mathcal{K}(P,P^*) = \mathcal A(P, P^*) = \mathcal N(P, P^*)$ by \cite{Jo}. 

Argyros and Haydon \cite{AH} constructed  Banach spaces $X_{AH}$ having a 
Schauder basis such that the Calkin algebra $\mathcal L(X_{AH})/\mathcal K(X_{AH})$ 
is one-dimensional. Their construction was subsequently modified by Tarbard \cite{Td}
to obtain certain finite-dimensional Calkin algebras. Dales \cite{D13} noted 
that it remains unclear whether 
there is a similar phenomenon for the compact-by-approximable algebras, that is, 
whether there are Banach 
spaces $X$ failing the A.P., such that $\A_X \neq \{0\}$ is a finite-dimensional (radical) algebra.
Note that if such a space $X$ exists, then $X$ cannot have the B.C.A.P. by 
Proposition \ref{bcapfact1}.(ii).

It also appears unknown whether there is a Banach space $X$ such that $\mathcal A(X)$
is separable (that is, $X^*$ is separable), but $\mathcal K(X)$ is non-separable.

\medskip

 \textit{Acknowledgements:}
We are grateful to Garth Dales for communicating \cite{D13} and for some subsequent discussions.
This paper is part of the  Ph.D.-thesis of Henrik Wirzenius under the 
supervision of the first author. H.W. gratefully acknowledges the financial support 
of The Swedish Cultural Foundation in Finland and the Magnus Ehrnrooth Foundation.

\bibliographystyle{amsalpha}
\bibliography{Tylli_Wirzenius_bib}

\end{document}